\documentclass{gtpart}    
\usepackage{pinlabel}  
\usepackage{graphicx}  
\usepackage[all]{xy}
\usepackage{amscd}
\usepackage{times}
\usepackage{enumerate}
\usepackage{indentfirst}
\usepackage{stmaryrd}
\usepackage{xypic}
\usepackage{mathtools}
\usepackage{euscript}

\author{Fedor Pavutnitskiy}
\givenname{Fedor}
\surname{Pavutnitskiy}
\address{Department of Mathematics, National University of Singapore, 10 Lower Kent Ridge Road, Singapore 119076}
\email{fedor@u.nus.edu}
\urladdr{}
\author{Jie Wu}
\givenname{Jie}
\surname{Wu}

\email{matwuj@nus.edu.sg}
\urladdr{math.nus.edu.sg/~matwujie}

  \title[Simplicial James-Hopf map and mod-$p$ UASS for suspensions]{Simplicial James-Hopf map and decompositions of the unstable Adams spectral sequence for suspensions}


\arxivreference{}
\arxivpassword{}

\newtheorem{thm}{Theorem}[section]  
\newtheorem{lem}[thm]{Lemma}        

\newtheorem{cor}[thm]{Corollary}

\theoremstyle{definition}
\newtheorem{defin}[thm]{Definition}   
\newtheorem*{rem}{Remark}             
\renewcommand{\Z}{\mathbb Z}

\renewcommand{\hom}{\mathrm{Hom}}

\renewcommand{\L}{\EuScript L}

\newcommand{\g}{\gamma}
\newcommand{\gp}{\gamma^{[p]}}
\newcommand{\ot}{\otimes}
\newcommand{\mono}{\: \ar@{^{(}->}}
\newcommand{\epi}{\ar@{->>}}
\renewcommand{\ker}{\mathrm{ker}\,}

\renewcommand{\lim}{\mathrm{lim}}
\newcommand{\im}{\mathrm{im}}

\newcommand{\xym}{\xymatrix}
\newcommand{\ra}{\, \Rightarrow\, }

\newcommand{\e}{\varepsilon}
\newcommand{\w}{\wedge}

\newcommand{\dlg}{\langle\langle}
\newcommand{\drg}{\rangle\rangle}
\renewcommand{\D}{\Delta}
\renewcommand{\top}{\mathrm{Top}_*}
\newcommand{\sset}{\mathrm{sSet}_*}
\renewcommand{\k}{\mathbf{k}}

\begin{document}

\begin{abstract}    
We use combinatorial group theory methods to extend the definition of a classical James-Hopf invariant to a simplicial group setting. This allow us to realize certain coalgebra idempotents at $\sset$-level and obtain a functorial decomposition of the spectral sequence, associated with the lower $p$-central series filtration on the free simplicial group.
\end{abstract}

\maketitle

\section{Introduction}

One of the ultimate goals in unstable homotopy is to understand how the loops-over-suspension functor $\Omega\Sigma$ changes the homotopy type of the space. One aspect of such understanding is the computation and investigation of the homotopy groups of loop spaces and one of the most powerful methods for it is an unstable Adams spectral sequence first introduced in \cite{6authors}. Construction of this spectral sequence starts with a free simplicial group model for a loop space, equipped with a lower $p$-central series filtration, which produces the spectral sequence itself. So framework for this machinery is a category of (pointed) simplicial sets $\sset$. Computations show that for a loop space over spheres, $p$-locally unstable Adams spectral sequence is accelerated in a sense that non-trivial elements on its first page are concentrated in columns with numbers $p^k$ (see \cite{6authors}).

Another approach to investigating loop-suspension functor $\top \to \top$ through functorial direct product decompositions was introduced in \cite{SelickWuCoalg}. It uses significantly a certain model for $\Omega\Sigma X$ in $\top$ - James construction $JX$ - a reduced free monoid on points of $X$ together with a word-length filtration $J_k X$ on it. Among all ($p$-local) decompositions of the form 
\begin{equation}\label{classicalDecomposition}
\Omega\Sigma X\simeq_p A\times B
\end{equation} there is a minimal one, called $A^{min}$ in a sense that it is obtained from a minimal coalgebra retract of a tensor algebra functor $T$ - which is homological representation of an $\Omega\Sigma$. It can be shown that primitive elements of such minimal coalgebra retract are concentrated in the degrees $p^k$.

The aim of the present paper is to translate the above mentioned decompositions to a category of simplicial sets and furthermore to extend them to a level of spectral sequences. From this point of view primitive elements of $A^{min}$ form a first page of a functorial sub-spectral sequence of unstable Adams spectral sequence for suspensions and therefore result about acceleration of spectral sequence for $\Omega S^{n+1}$ becomes a particular case (since $A^{min}(S^n)=S^n$) of a more general situation.

We will briefly recall the idea of the loop-suspensions decompositions as in \cite{SelickWuCoalg}. Let $\k=\Z/p$ be a ground field, then by Bott-Samelson theorem homology of $\Omega\Sigma X$ with coefficients in $k$ is given by the tensor algebra $T(\bar H_*(X))$ on the reduced homology on $X$, i.e. functorially there is a map
\begin{equation}\label{classicalRealization}
[\Omega\Sigma,\Omega\Sigma]\xrightarrow{H_*}\hom_{\mathrm{Coalg}}(T,T)
\end{equation}
By geometric realization theorem in \cite{SelickWuCoalg} this map is surjective, and in fact all natural coalgebra maps $f:T\to T$ can be realized as an elements in a certain subgroup of $[J,\Omega\Sigma]$ called Cohen group $H_{\infty}$ first introduced in \cite{CohenComb}. Now given any coalgebra idempotent $f:T\to T$ it can be lifted to a space level to obtain a map $\varphi: \Omega\Sigma\to \Omega\Sigma, \ \varphi_*=f$ and $A=\mathrm{hocolim}\,\varphi$ will give a first piece of decomposition (\ref{classicalDecomposition}), second piece is given similarly by a homotopy mapping telescope of a complement map $[\mathrm{id}]\varphi^{-1}$.

To adapt this picture to $\sset$, one need a suitable analogue of the Cohen group $H_\infty$ which will realize a coalgebra idempotents on the level of simplicial sets. In $\top$ Cohen group is defined as a subgroup of $[J,\Omega\Sigma]$, generated as a pro-group by compositions 
\begin{equation}\label{classicalGenerators}
JX\xrightarrow{H_k}JX^{\w k}\xrightarrow{J(\bar\Delta)}JX^{\w l}\xrightarrow{J(\sigma)}J(X^{\w l})\xrightarrow{W_l}\Omega\Sigma X
\end{equation}
first and a last map here is combinatorial James-Hopf map and Whitehead product respectively (see definition (\ref{classicalJH}) below). The most straightforward way to translate $H_\infty$ to a $\sset$ is to pass all objects through geometric realization/singular chains functors. Unfortunately, heaviness construction of $|-|-\mathrm{Sing}$ adjunction hides combinatorial nature of $H_{\infty}$, so preservation of lower $p$-central series by its elements looks untraceable under this approach.

Instead we will define generators of Cohen group as a certain natural transformations of functors on pointed sets and will extend this definition to a $\sset$ level-wisely. The main object of interest here is Milnor's construction $F[-]$ - simplicial group model for a loop-suspension functor $\Omega\Sigma$. Our first goal is to extend the combinatorial James-Hopf map from free monoid to a free group framework. It is achieved by translating Hilton-Hopf map to simplicial level and applying Hall's commutator collection process inside simplicial version of Hilton-Milnor's theorem. The resulting definition (\ref{newJH}) have all desired properties of a classical James-Hopf map and moreover can be considered as a non-abelian version of Magnus embedding \cite{MagnusRepr}. Since everything happens level-wisely, Whitehead product can be considered just as iterated commutator map, permutations and iterated reduced diagonal maps stays the same as in (\ref{classicalGenerators}).

With help of Fox differential calculus it will be shown that elements of Cohen group preserve lower central series on free simplicial group $F[-]$ and therefore induce a map between lower ($p$-)central towers of fibrations. Extending these maps to an exact couple level will equip unstable Adams spectral sequence with an action of Cohen group.

Such action used to obtain a functorial decomposition of this spectral sequence in a following way. As before, one starts with a functorial coalgebra idempotent $f:T\to T$, which is now extended to a natural self-transformation of functor $T(\k[-]):\sset\to\mathrm{sCoalg}_{\k}$. By simplicial realization theorem (\ref{simplicialRealizationTheorem}) this natural transformation can be lifted to a self-map of l.c.s. tower and each level of the corresponding tower can be decomposed through a mapping telescope construction in a usual way. Such decomposition preserves the fibrations inside the l.c.s. tower, which can be therefore presented as a direct product of two towers, with layers weak equivalent to primitive elements of coalgebra retracts of $T(\k[-])$, corresponded to idempotent $f$ and its complement (theorem (\ref{towerDecomposition}). Choosing an idempotent which corresponds to a minimal coalgebra retract ($A^{min}$ subfunctor of $T$) will produce a direct summand of l.c.s. tower with non-trivial layers only at degrees $p^k$. Therefore corresponding spectral sequence will be accelerated in the same sense as in \cite{6authors}.

The paper is organized as follows. Extension of the James-Hopf invariants to the simplicial group level (definition (\ref{newJH})) together with motivation of such extension and connection with Magnus embedding (theorem (\ref{JHandMagnus})) is given in Section 2. Preservation of lower $p$-central series by elements of Cohen group discussed in theorem (\ref{JHLowerCentral}) and the construction itself is given in definition (\ref{CohenGroupGenerators}) of Section 3. Section 4 is devoted to applications of action of Cohen group on spectral sequence (\ref{actionOnSS}): decomposition of spectral sequence (\ref{towerDecomposition}) and (\ref{SSDecompositionCorollary})) and acceleration of the part that correspond to a minimal coalgebra decomposition (\ref{acceleration})

\section{James-Hopf map for Milnor's construction}
$\mathrm{Set}_*$ will denote the category of pointed sets (with distinct point $*$). In this category \textit{wedge sum} $\xym{X\vee Y=X\cup Y/*_X\sim *_Y\mono[r]& X\times Y}$ and \textit{smash product} $X\w Y=\frac{X\times Y}{X\vee Y}$ of pointed spaces $X$ and $Y$ are defined. Category of groups $\mathrm{Grp}$ will also be considered pointed, with identity element of each object as a distinct point. Then there is a pair of adjoint functors 
\begin{equation*}
F:\mathrm{Set}_*\rightleftarrows \mathrm{Grp}:U
\end{equation*}
where $U$ is forgetful functor and $F$ is a \textit{reduced} free group functor, that's it, for any pointed set $X$ (with distinct point $*$), $FX$ is a free group on $X$, quotient out by relation $*=1$. The corresponding monad $UF:\mathrm{Set}_*\to \mathrm{Set}_*$ and its extension to the category of pointed simplicial sets $\sset$ will be denoted by $F[-]$ and called \textit{Milnor's construction}.

The similar monad, obtained from adjoint functors between $\mathrm{Set}_*$ and category of monoids $\mathrm{Mon}$ will be denoted by $J$ - \textit{James construction}, which form a subfunctor of $F[-]$.
\begin{rem}
Traditionally, James construction is defined for topological space as a free reduced topological monoid on points of topological space. In the present paper James construction applied to simplical set levelwise. Moreover, various functors and natural transformations between them will be considered for $\mathrm{Set}_*$ and $\sset$ simultaneously, whenever it can not lead to the confusion.
\end{rem}
\begin{defin}\label{classicalJH}
The $k$-th \textit{(classical) combinatorial James-Hopf map} is a natural transformation of functors
\begin{equation*}
H_k: J\to J(-)^{\w k}
\end{equation*}
such that for any pointed set $X$, $H_k: JX\to JX^{\w k}$ defined as
\begin{equation*}
H_k(x_1\dots x_n)=\prod_{i_1<\dots<i_k}x_{i_1}\w \dots \w x_{i_k}
\end{equation*}
where product is taken over all subsequences in $1..n$ (without repetitions of indices) arranged in the (right) lexicographic order. By functoriality the same map can be defined for simplicial sets levelwise.
\end{defin}

The aim of this section is to extend the definition (\ref{classicalJH}) in a natural way to the natural transformation $F[-]\to F[(-)^{\w n}]$ and give a motivation for such definition along with some basic properties. We will denote this extension by the same letter $H_n$ whenever it cannot lead to a confusion. 
\begin{defin}
The $n$-th \textit{combinatorial James-Hopf} map is a natural transformation
\begin{equation*}
H_n: F[-]\to F[(-)^{\w n}]
\end{equation*}
defined for any pointed (simplicial) set $X$ on reduced words as 
\begin{equation*}\label{newJH}
H_k(x_1^{\e_1}\dots x_n^{\e_n})=\prod_{(i_1\dots i_k)} (x_{i_1}\w\dots\w x_{i_k})^{\e_{i_1}\dots \e_{i_k}}
\end{equation*}
here product is taken in (right) lexicographical order over sequences of indices $(i_1\dots i_k)$ such that 
$$
i_j\leq i_{j+1}-\frac{\e_{i_{j+1}}+1}{2}
$$
that's it, product is taken over all subsequences $(i_1\dots i_k)$ of $(1\dots n)$ with possible repetition of indices and repetition of index $i_j$ occurs only if corresponding exponent $\e_{i_j}$ is negative.
\end{defin}
Two immediate remarks here:
\begin{enumerate}[\upshape(i)]
\item By construction (see commutator collection process below) $H_k$ is well-defined for all (non necessary reduced) words in $F[X]$, for example:
\begin{equation*}
*=H_k(*)=H_k(x x^{-1})=(x\w\dots\w x)^{(-1)^{k-1}}(x\w\dots\w x)^{(-1)^k}=*
\end{equation*}
\item $H_k$ is a natural extension of a combinatorial James-Hopf map for James construction:
\begin{equation}\label{connectionWithClassicalJH}
\xym{
F[X]\ar[r]^{H_k}&F[X^{\wedge k}]\\
J(X)\mono[u]\ar[r]^{H_k}&J(X^{\wedge k}\mono[u])
}
\end{equation}
\end{enumerate}

To motivate this definition we follow the idea that Curtis apparently used in \cite{CurtisSimpl} to define a Hopf map $H_2$ in simplicial EHP sequence.

In $\top$ there is an another map $h_k: \Omega\Sigma X\to \Omega\Sigma X^{\w k}$ called \textit{Hilton-Hopf invariant} which is very close to a classical James-Hopf invariant $H_k$ (in a sense that it induces the same map on homology and can be used in EHP-sequence), but defined only for co-$H$-spaces $X$, see \cite{Neisen}. We will use an extension of this definition to whole $\top$ and afterwards translate it to $\sset$. To define $h_k$ first recall the statement of the Hilton-Milnor's theorem as in \cite{Neisen}. For connected spaces $X$ and $Y$ loop-suspension over their wedge can be decomposed as a product
\begin{equation}\label{HiltonMilnor}
\Omega\Sigma (X\vee Y)\simeq \Omega\Sigma X\times \Omega\Sigma(\bigvee_{n\geq 0} Y\w X^{\w n})
\end{equation}
The second summand can be decomposed further and iterated application of this theorem gives the following decomposition
\begin{equation}\label{HiltonMilnorIterate}
\Omega\Sigma(X\vee Y)\simeq \prod_{\omega\in \EuScript B (X,Y)}\Omega\Sigma (\omega(X,Y))
\end{equation}
here $\EuScript B (X,Y)$ denotes the Hall basis of free Lie algebra on two letters $X$ and $Y$ and so each summand $\omega(X,Y)$ can be identified with a certain commutator bracket on two letters, written down as a smash product.

Let $X$ be a co-$H$-space with comultiplication $\mu'$. For the matter of convenience we will denote by $A$ and $B$ two identical copies of $X$ With this notation comultiplication will be a map $\mu': X\to A\vee B$ and Hilton-Hopf map $h_k$ is defined as a composition
\begin{equation}\label{HiltonHopfDefinition}
\xym{
\Omega\Sigma X\ar[d]_{\Omega\Sigma \mu'}\ar[r]^{h_k}&\Omega\Sigma X^{\w k}\\
\Omega\Sigma(A\vee B)\ar[r]^-{\sim}& \prod_{\omega\in \EuScript B (A,B)}\Omega\Sigma(\omega(A,B))\ar[u]_{\pi_{B\w\dots\w A}}
}
\end{equation}
here lower horizontal map is an inverse of the weak equivalence in Hilton-Milnor's theorem (\ref{HiltonMilnorIterate}), right vertical map is a projection to summand, corresponded to Lie bracket $[B,A]\dots],A]$.

To extend $h_k$ to whole $\top$ we will introduce another map $\Omega\Sigma X\to \Omega\Sigma (X\vee X)$, as in \cite{CurtisSimpl}, which is homotopy equivalent to $\Omega\Sigma \mu'$ in case of co-$H$-space $X$. Let $i_A$ and $i_B$ be two inclusions of $X$ in $\Omega\Sigma(A\vee B)$, then $i_A * i_B$ is defined as a product of these inclusions in $\Omega\Sigma(A\vee B)$, i.e. extension of the composition
\begin{equation}\label{HHMapComposition}
X\xrightarrow{\Delta} X\times X\xrightarrow{i_A\times i_B}\Omega\Sigma(A\vee B)\times \Omega\Sigma(A\vee B)\xrightarrow{\mu} \Omega\Sigma (A\vee B)
\end{equation}
to an $H$-map $\Omega\Sigma X\to \Omega\Sigma A\vee B$
\begin{lem}\label{HiltonHopfExtension}
For co-$H$-space $X$ 
\begin{equation*}
i_A*i_B \simeq \Omega\Sigma \mu'
\end{equation*}
\end{lem}
\begin{proof}
Since extension of (\ref{HHMapComposition}) to the $H$-map is unique up to homotopy, the statement follows from the commutativity of the following diagram, which can be checked explicitly:
\begin{equation*}
\xym{
X\ar[rd]_{\mu'}\ar[r]^-{\Delta}& X\times X\ar[r]^-{i_A\times i_B}&\Omega\Sigma(A\vee B)\times \Omega\Sigma(A\vee B)\ar[r]^-{\mu}& \Omega\Sigma (A\vee B)
\\
&X\vee X\mono[u]\mono[urr]
}
\end{equation*}

\end{proof}
Now we can translate definition (\ref{HiltonHopfDefinition}) to a simplicial group setting:
\begin{equation}\label{motivationOfJH}
\xym{
F[X] \ar[d]_{i_A*i_B}\ar[r]^{H_k}& F[X]^{\w k}\\
F[A]* F[B]\ar[r]^-{\simeq}& \prod_{\omega\in \EuScript B (A,B)}F[\omega(A,B)]\ar[u]_{\pi_{B\w A\w\dots \w A}}
}
\end{equation}
here  $*$ denotes free product $F[A]*F[B]\cong F[A\vee B]$ and left vertical map sends generators $x_i$ to $a_i b_i$ - product of generators of $F[A]$ and $F[B]$.

In $\top$ Hilton-Milnor's decomposition given by a certain weak equivalence $\Omega\Sigma X\times \Omega\Sigma(\bigvee_{n\geq 0} Y\w X^{\w n})\to \Omega\Sigma (X\vee Y)$ is constructed from iterated Whitehead products $[i_Y,i_X]\dots],i_X]$, so inverse to it (lower horizontal arrow) is somewhat mysterious. The situation is opposite in $\sset$: the map from free product $F[A] * F[B]\to \prod_{\omega\in \EuScript B (A,B)}F[\omega(A,B)]$ can be seen as a collection process in sense of Hall (see \cite{Hall}), this allow us to describe the map purely combinatorially and it is turns such map is a natural generalization of combinatorial James-Hopf map from free monoid to a free group. We briefly recall procedure of commutator collection applied to a free product $F[A] * F[B]$ from \cite{CurtisSimpl}. First for simplicity we assume that $A$ (elements denoted by $a_i$) and $B$ (elements denoted by $b_j$) are just pointed sets, after that the process can be extended to a simplicial sets by just applying it level-wise.

There is a natural projection $F[A]*F[B]\to F[A]$ (forgetting all $b_j$) with a section, induced by inclusion $A\to A\vee B$ and by Schreier's lemma the kernel is given by $F[B]*F[B\w F[A]]$, here free group on wedge $B\w F[A]$ is naturally identified as a subgroup of $F[A] * F[B]$ generated by commutators $[b,w] \ b\in B,\, w\in F[A]$ .The free summand $F[B\w F[A]]$ can be decomposed further using Tietze transformations ($m$ is any number)
\begin{align}\label{Tietze}
F[B\w F[A]]\cong F[B\w A]*F[B\w A\w F[A]]\cong\dots\\ \nonumber \dots\cong F[B\w A]*F[B\w A\w A]*\dots * F[B\w A^{\w m}\w F[A]]
\end{align}
this chain of isomorphisms can be seen as a process of expanding commutators of the form $[b,aw]$ as $[b,w][b,a][b,a,w], \ b\in B, \, a\in A, \, w\in F[A]$
So, for any $m$ after forgetting group structure there is a map 
\begin{align}\label{HMbijection}
F[A]*F[B]\to F[A]\times F[B]*F[B\w A]*\dots \\ \nonumber \dots * F[B\w A^{\w m}]* F[B\w A^{\w m}\w F[A]]\to F[A]\times F[\bigvee_{j=0}^m B\w A^{\w j}]
\end{align}
given by (set-theoretic) retraction $F[A]*F[B]\to F[B]*F[B\w F[A]]$, composed with a chain of isomorphisms (\ref{Tietze}) and projection away from $F[B\w A^{\w m}\w F[A]]$. Hilton-Milnor's theorem for simplicial sets states that this map becomes a weak equivalence as $m\to \infty$ since connectivity of pieces $F[B\w A^{\w m}\w F[A]]$ grows as $m$ increases. Same procedure can be applied to $F[B]*F[B\w F[A]]$ to split off $F[B]$ as direct summand and so on. Sequence of such procedures can be seen as Hall's commutator collection process which described in details in \cite{Hall} and \cite{CurtisSimpl}. Here we only sketch the basic idea.

Let $w=a_1^{\e_1}b_1^{\eta_1}\dots a_n^{\e_n}b_n^{\eta_n}$ be a reduced word in $F[A]*F[B], \ \e_i,\, \eta_j=\pm 1$. For simplicity first assume that all $\e_i=\eta_i=1$. By pushing all $a_i$ in the beginning of the word $w$, starting with $a_1$, using identity 
$$
c a_i=a_i c [c,a_i], \ c=b_j \ \mathrm{or} \ c=[b_j,a_{i_1}\dots a_{i_k}]
$$

word $w$ can be written as $w=w_A w'$, where $w_A=a_1^{\e_1}\dots a_n^{\e_n}$ and $w'$ consist of $b_j$ and commutators of the form $[b_j,a_{i_1}]\dots ]a_{i_k}]$, so $w'\in F[B\w F[A]]$ and mapping $w\to w'$ gives a retraction $F[A]*F[B]\to F[B]*F[B\w F[A]]$. 
In case of arbitrary exponents the following commutation rules should be applied:
\begin{align}
&c^{-1}a=a[c,a]^{-1}c^{-1}\label{firstRule}\\
&c a^{-1}=a^{-1}c[c,a^{-1}]=a^{-1}c[c,a]^{-1}[c,a,a^{-1}]^{-1}=\dots\label{secondRule}\\
&c^{-1} a^{-1}=a^{-1}[c,a^{-1}]^{-1}c^{-1}=a^{-1}[c,a,a^{-1}][c,a]c^{-1}=\dots\label{thirdRule}
\end{align}
in the second and third rule here commutator identity $[x,y^{-1}]=[x,y]^{-1}[x,y,y^{-1}]^{-1}$ is applied to take negative exponent out of commutator (to move from $F[B\w F[A]]$ to $F[B\w A]*F[B\w A\w F[A]]$). Also note that exponents of commutators that lies inside $F[B\w A]$ are products of exponents of $c$ and $a$.

Described process collect (arrange in (right) lexicographic order in the beginning of the word $w$) commutators of weight 0 (i.e. generators $a_i$), next iteration of Hilton-Milnor's theorem will collect  all $b_j$, $[b_j,a_i]$ and so on. Now, using retraction mentioned above, we are ready to describe composition (\ref{motivationOfJH}).
\begin{rem}
Since in (\ref{motivationOfJH}) the last map is a projection to summand $F[B\w A\w \dots \w A]$ (taking commutators of the form $[b_j,a_{i_1}\dots a_{i_k}]$), to define $H_n$ only one retraction $w \to w'$ is needed:
\begin{equation}\label{oneIter}
\xym{
F[X]\ar[d]_{i_1 * i_2}\ar[rr]^{H_n}&&F[X^{\w n}]\\
F[A\vee B]\ar[r]^-{\rho}&F[\bigvee_{j=0}^kB\w A^{\w j}]*F[B\w A^{\w k} \w F[A]]\ar[r]&\prod_{j=0}^k F[B\w A^{\w j}]\times F[B\w A^{\w k}\w F[A]]\ar[u]_{\pi_{B\w A\dots\w A}}
}
\end{equation}
Difference between (\ref{motivationOfJH}) and (\ref{oneIter}) is in order of commutators in each $F[B\w A^{\w k}]$. To ensure that new definition of $H_k$ will contain old one for words with positive exponents, commutators needed to be arranged in (right) lexicographical order, which happens in further applications of Hilton-Milnor's decomposition. So in some sense (which will be made precise, see theorem (\ref{JHandMagnus}) below) $H_k$ defined uniquely only up to order of elements in the product in $F[B\w A^{\w k}]$
\end{rem}

Fix $k>0$ and consider $H_k(x_1^{\e_1}\dots x_n^{\e_n})$ as in (\ref{motivationOfJH}). The image of $x_1^{\e_1}\dots x_n^{\e_n}\in F[X]$ under $i_1*i_2$ is $a_1^{\e_1}b_1^{\e_1}\dots a_n^{\e_n}b_n^{\e_n}\in F[A]*F[B]$. During collection process each $a_i^{\e_i}$ will give rise to commutators of the form $[b_j,a_{i_1}\dots a_{i_l},a_i]^{\e_j \e_{i_1}\dots \e_{i_l}\e_i}$, where $j\leq i_1\leq \dots \leq i_l\leq i$ and equality in this sequence appears only if corresponding exponent is $-1$ (rules (\ref{secondRule}) or (\ref{thirdRule}) applied). After arranging such commutators for all $i$ lexicographically and taking that of length $k$ we get a definition (\ref{newJH}).

Now we will show that product of all $H_k$ can be seen as a non-abelian version of Magnus representation \cite{MagnusRepr}. Let $X$ be a pointed set and $\Z\dlg X\drg$ be a ring of (non-commutative) formal power series on elements of $X$. \textit{Magnus representation} $\mu$ is a homomorphism from free group on $X$ to $\Z\dlg X\drg$ defined on generators as
\begin{equation*}
\mu : F[X]\to \Z\dlg X\drg, \  \ x \mapsto x+1
\end{equation*}
By theorem of Magnus \cite{MagnusRepr} this map is injective.

Ring $\Z\dlg X\drg$ can be identified with an infinite product ($\bar \Z[-]=\Z[-]/\Z[*]$)
$$
\prod_{n=0}^{\infty}\bar\Z[X]^{\otimes n}\cong \prod_{n=0}^{\infty}\bar\Z[X^{\w n}]
$$
and for any $k$ there are projections $\pi_k:\Z\dlg X\drg\to \bar\Z[X^{\w n}]$ which maps formal power series to its elements of length $k$. By \textit{length} of formal power series we mean the minimum among all lengths of its elements.
\begin{thm}\label{JHandMagnus}
there is a commutative diagram
\begin{equation}\label{JHandMagnusDiagr}
\xym{
F[X]\mono[r]^{\mu}\ar[d]_{H_k}& \Z\dlg X\drg\ar[d]^{\pi_k}\\
F[X^{\w k}]\epi_{\mathrm{ab}}[r]&\bar\Z[X^{\w k}]
}
\end{equation}
\end{thm}
\begin{proof}
Define $H_0:F[X]\to \Z, \, x\mapsto 1$ and $H=\prod_{k=0}^{\infty}:F[X]\to \prod_{k=0}^{\infty}F[X^{\w k}]$. Then the statement follows from commutativity of the following diagram
$$
\xym{
F[X]\mono[r]^{\mu}\ar[d]_H&\Z\dlg X\drg\\
\prod_{k=0}^{\infty}F[X^{\w k}]\ar[ur]_{\mathrm{ab}}
}
$$
For any $w=x_1^{\e_1}\dots x_n^{\e_n}\in F[X]$ coefficients of power series $\mu(w)\in\Z\dlg X\drg$ can be described as augmentations of Fox derivatives (see \cite{FoxFDCI})
\begin{equation}\label{MagnusFox}
\mu(w)=\sum_{k=0}^{\infty}\sum_{(i_1\dots i_k)}\e\partial^k_{i_1\dots i_k}(w)x_{i_1}\dots x_{i_k}
\end{equation}
here we use shorten notation $\partial^k_{i_1\dots i_k}(w)$ for Fox derivative $\frac{\partial^k w}{\partial x_{i_1}\dots \partial x_{i_k}}$. Each of these given by formula ((3.5) in \cite{FoxFDCI})
$$
\partial^k_{i_1\dots i_k}(w)=\sum_{(\lambda_1\dots\lambda_k)}\e_{\lambda_1}\dots \e_{\lambda_k}
$$
with summation over all sequences of indices $(\lambda_1\dots\lambda_k)$ such that $i_{\lambda_j}=i_j$ for all $j=1\dots k$ and
$$
\lambda_j\leq \lambda_{j+1}-\frac{\e_{\lambda_{j+1}}+1}{2}
$$
comparison it with definition (\ref{newJH}) gives the result.
\end{proof}
\begin{cor}
\begin{enumerate}[\upshape(i)]
\item $H:F[X]\to \prod_{k=0}^{\infty}F[X^{\w k}]$ is injective
\item Let $\gamma_n$ be a $n$-th term of lower central series of $F[X]$ and $\L^n(X)\cong\gamma_n/\gamma_{n+1}$ be a $n$-th Lie power, viewed as $n$-th homogeneous component of associated graded object, corresponded to lower central series filtration on $F[X]$. $\bar H_n$ will denote composition of $H_n$ with abelianization map $F[X]^{\w n}\to \bar \Z[X]^{\otimes n}$. Then the following diagram commute
\begin{equation*}
\xym{
\gamma_n\mono[r]\epi[d]&F[X]\ar[d]^{\bar H_n}\\
\L^n(X)\mono[r]&\bar\Z[X]^{\otimes n}
}
\end{equation*}
i.e. $\bar H_n$ sends group commutators to ring commutators
\end{enumerate}
\end{cor}
\begin{proof}
first statement is immediate, second one follows from the fact that elements of length $n$ in $\mu(\gamma_n)$ forms $\L^n(X)$, see \cite{MagnusRepr}.
\end{proof}
\section{Cohen group for the lower $p$-central series tower}
In this section we define a Cohen group as a subgroup of natural transformations of functors $F[-]\to F^{\w}[-]$ which acts naturally on the lower $p$-central series tower, and therefore on the corresponding spectral sequence.

Let $X$ be a (pointed) simplicial set and $p$ be a prime. Recall that lower central series $\gamma_n$ and lower $p$-central series $\gp_n$ of $F[X]$ defined as a simplicial subgroups of $F[X]$ such that in each dimension $k$ they are usual lower ($p$-)central series of a group $F[X]_k$:
\begin{align}
(\gamma_n F[X])_k=\langle\{[g_1\dots g_n] \ | \ g_l\in F[X_k]\}\rangle\\
(\gp_n F[X])_k=\langle\{[g_1\dots g_i]^{p^j}\ | \ g_l \in F[X_k], \ ip^j\geq n\}\rangle
\end{align}
$\gamma_n$ and $\gp_n$ will be considered as a subfunctors of $F[-]$.

Both series leads to a (functorial with respect to $X$) towers of fibrations of simplicial sets (denoted $\Gamma(X)$ and $\Gamma_p(X)$ respectively):
\begin{equation}\label{towers}
\xym{
F^{\w}[X]=\varprojlim F[X]/\g_n\ar[d]&&F_p^{\w}[X]=\varprojlim F[X]/\gp_n\ar[d]&\\
\vdots\ar[d]&&\vdots\ar[d]&\\
F[X]/\g_{n+1}\epi[d]&\L^n(\bar\Z[X])\mono[l]&F[X]/\gp_{n+1}\epi[d]&\L_{res}^n(\bar\Z/p[X])\mono[l]\\
F[X]/\g_n\epi[d]&&F[X]/\gp_n\epi[d]&\\
\vdots\epi[d]&&\vdots\epi[d]&\\
F[X]/\g_2=\bar\Z[X]&&F[X]\gp_2=\bar\Z/p[X]
}
\end{equation}

Fibers of (\ref{towers}), denoted by $E^0\Gamma$ and $E^0\Gamma_p$ can be identified with Lie powers and restricted Lie powers on $X$ and both towers leads to spectral sequences (convergence is due to Curtis \cite{CurtisSS})
\begin{equation}\label{spectralSeq}
E^1_{p,q}=\pi_p \L^q(\bar\Z[X])\ra \pi_{p+q}F[X], \ E^1_{p,q}=\pi_p \L^q(\bar\Z/p[X])\ra \pi_{p+q}F[X]\otimes \Z/p
\end{equation}
\begin{rem}
Traditionally spectral sequence, constructed from lower $p$-central series filtration on Kan construction $GX$ is called \textit{unstable Adams spectral sequence} and was introduced in \cite{6authors} (accelerated version) and \cite{Rector}. Spectral sequence (\ref{spectralSeq}) is isomorphic to an unstable Adams spectral sequence for $G(\Sigma X)$ by a natural isomorphism of simplicial groups $F[X]\to G(\Sigma X)$.
\end{rem}

(Levelwise) \textit{Whitehead product} $W_n$ for Milnor's construction is a homomorphism (a natural transformation, in fact) $W_n:F[X^{\w n}]\to F[X]$ defined on generators as
\begin{equation*}
W_n(x_1\w\dots\w x_n)=[x_1,x_2]\dots],x_n]
\end{equation*}

Now we can construct an analogue of Cohen group (\ref{classicalGenerators}) for simplicial sets as a subgroup of $\hom_{\sset}(F[-],F^{\w}[-])$ - (topological) group (under pointwise multiplication) of all simplicial natural transformations between Milnor's construction and its pro-nilpotent completion. Reason of choosing pro-nilpotent completion here is that original Cohen group $H_{\infty}$ was constructed as a pro-group and therefore contains certain infinite products, which can be crucial for realizing coalgebra idempotents. 

\begin{defin}
\textit{Cohen group} $ H_{\infty}$ is a closure of a subgroup of $\hom_{\sset}(F[-],F^{\w}[-])$ generated by compositions (evaluated at arbitrary $X$)
\begin{equation}\label{CohenGroupGenerators}
F[X]\xrightarrow{H_k}F[X^{\w k}]\xrightarrow{F[\bar\Delta]}F[X^{\w l}]\xrightarrow{F[\sigma]}F[X^{\w l}]\xrightarrow{W_l}F[X]\to F^{\w}[X]
\end{equation}
for all positive $k\leq l$ and permutations $\sigma\in\Sigma_l$. Here $H_k$ is James-Hopf map (\ref{newJH}), $\bar\Delta: X^{\w k}\to X^{\w l}$ is iterated reduced diagonal, $\sigma:X^{\w l}\to X^{\w l}$ permuting factors in smash product and $W_l$ is a Whitehead product.
\end{defin}

To check that elements of $H_{\infty}$ induces maps between towers $\Gamma\to \Gamma$ and $\Gamma_p\to \Gamma_p$ it is sufficient to check that generators (\ref{CohenGroupGenerators}) preserves $\g_n$ and $\gp_n$ and since $F(\bar\Delta), \, F(\sigma)$ and $W_l$ are homomorphisms only preservation of lower ($p$-)central series by $H_k$ needed to be checked. Unfortunately, this is not a case as simplest example shows:
\begin{equation*}
H_2([x,y])=[(x\w x)^{-1},(x\w y)^{-1}][(y\w y)^{-1},(y\w x)](y\w x)^{-1}(x\w y)
\end{equation*}
To overcome this difficulty we define weighted lower central series filtration on $F[X^{\w m}]$ and show that James-Hopf map preserve lower ($p$-)central series in this weighted sense, which is sufficient to induce a maps of towers $\Gamma$ and $\Gamma_p$.
\begin{defin}
\textit{Weighted lower ($p$-)central series }$\g^w_n (F[X^{\w m}])$ and $\gamma^{[p], w}_n (F[X^{\w m}])$ are sequences of simplicial subgroups of $F[X]$ defined as follows:
\begin{enumerate}[\upshape(i)]
\item if $n<m$ then $\g^w_n (F[X^{\w m}])=F[X^{\w m}]$
\item if $n\geq m$. i.e. $n=qm+s$, $s<m$
\begin{equation*}
\g^w_n (F[X^{\w m}])=\g_{\lceil \frac{n}{m}\rceil} (F[X^{\w m}])=\begin{cases} \g_q(F[X^{\w m}]), & \mbox{if } s=0 \\ \g_{q+1}(F[X^{\w m}]), & \mbox{if } s>0 \end{cases}
\end{equation*}
and similarly for $\gamma^{[p], w}_n (F[X^{\w m}])$
\end{enumerate}
\end{defin}
\begin{lem}\label{WhiteheadLowerCentral}
Whitehead products sends weighted lower ($p$-)central series to the usual one
\begin{equation*}
W_m(\g^w_n(F[X^{\w m}]))\subset \g_n, \ W_m(\gamma^{[p], w}_n (F[X^{\w m}]))\subset \gp_n
\end{equation*}
\end{lem}
\begin{proof}
Statement follows from the well known fact that $[\g_n,\g_m]\subset \g_{n+m}$. Mod-$p$ version follows from the integral one:
$$
W_m(\prod_{(i,j):ip^j\geq \lceil\frac{n}{m}\rceil}(\gamma_i)^{p^j})=\prod_{(i,j):ip^j\geq \lceil\frac{n}{m}\rceil}W_m(\gamma_i)^{p^j}\subset\prod_{(i,j):ip^j\geq \lceil\frac{n}{m}\rceil}(\gamma_{im})^{p^j}\subset\prod_{(i,j):i'p^j\geq n}(\gamma_{i'})^{p^j}
$$
where $i'=im$ in the last product
\end{proof}

Since we are working with free (simplicial) groups, it is convenient to switch our attention from lower ($p$-)central series $\gamma_n$ and $\gp_n$ to powers of (mod-$p$) augmentation ideals $\Delta^n$ and $\Delta_p^n$. Recall that \textit{augmentation ideal} and \textit{(mod-$p$) augmentation ideal} are kernels of augmentation homomorphisms
\begin{equation*}
\xym{\Delta\mono[r]&\Z[F[X]]\epi[r]^-{\e}&\Z}, \ \xym{\Delta_p\mono[r]&\Z/p[F[X]]\epi[r]^-{\e}&\Z/p}
\end{equation*}
According to Magnus-Witt theorem $\g_n$ and $\gp_n$ can be expressed through $\Delta^n$ and $\Delta_p^n$
\begin{equation*}
\g_n=F[X]\cap(1+\Delta^n), \ \gp_n=F[X]\cap(1+\Delta^n_p)
\end{equation*}
We will use the following lemma (stated simultaneously for integral and mod-$p$ case):
\begin{lem}\label{lowerCentralAugmentationLemma}
Let $f:A\to B$ be a pointed ($f(1)=1$) map between free groups $A$ and $B$ such that $\Z f(\D^n)\subset\D^m$ ($\Z/p f(\D^n_p)\subset\D^m_p$). Then $f(\gamma_n)\subset \gamma_m$ (resp. $f(\gp_n)\subset \gp_m$) and there is a well defined map $\bar f: A/\gamma_n\to B/\gamma_m$ (resp. $\bar f: A/\gp_n\to B/\gp_m$)
\end{lem}
\begin{proof}
We will prove only integral case. Using identification $\gamma_n=A\cap(1+\D^n)$,  $\forall g\in\gamma_n$ $g=1+a, \, a\in \D^n$. So $f(g)=\Z f(1+a)=1+\Z f(a)\in (1+\D^m)\cap B=\gamma_m$. For the second statement it is sufficient to show that $\forall g\in A, x\in\gamma_n \ f(gx)\in f(g)\gamma_m$. Using the same identification $x=1+b, \, b\in\D^n\ra f(gx)=f(g)+\Z f(gb)$, so $f(gx)-f(g)\in \D^m$, $f(g)^{-1}f(gx)-1\in \D^m$, and finally $f(g)^{-1}f(gx)\in \gamma_m$
\end{proof}

The standard machinery for working with powers of augmentation ideal is a Fox free differential calculus, first described in \cite{FoxFDCI}. Here we will briefly sketch the definition of Fox derivatives and state the theorem of Fox, connecting these derivatives and powers of augmentation ideal.
\begin{defin}
Let $FX$ be a free group on the set of generators $X$, $x_i\in X$ \textit{Fox derivative with respect to $x_i$} is a linear map
\begin{equation*}
\partial_i:\Z [FX]\to\Z [FX]
\end{equation*} 
uniquely determined by the following properties
\begin{align*}
&\partial_i(x_j)=\begin{cases} 1, & \mbox{if } i=j \\ 0, & \mbox{if } i\neq j \end{cases}\\
&\partial_i(ab)=\partial(a)\e(b)+a\partial(b), \ a,b\in \Z [FX]
\end{align*}
\end{defin}
For a sequence of generators $x_{i_1}\dots x_{i_k}$, $\partial^k_{i_1\dots i_k}$ will denote higher Fox derivative $\partial_{i_1}\circ\dots\circ \partial_{i_k}$ of order $k$. Fox derivative of order $0$ is defined to be an augmentation homomorphism $\e$.
\begin{thm}[Fox, \cite{FoxFDCI}]\label{FoxTheorem}
Let $a\in \Z [FX]$, then $a\in \Delta^n$ if and only if for all $0\leq k<n$ and for all sequences of indices $(i_1\dots i_k)$, $\e\partial^k_{i_1\dots i_k}(a)=0$, i.e. element of a group ring lies in the $n$-th power of augmentation ideal iff augmentations of all its Fox derivatives up to order $n-1$ vanishes.
\end{thm}
\begin{rem}
This theorem can be seen as a corollary of formula (\ref{MagnusFox}) above, since image of $\Delta^n$ under Magnus embedding (extended linearly to a group ring) is $P_n\subset\Z\dlg X\drg$ - formal power series of length $\geq n$.
\end{rem}
\begin{cor}\label{FoxTheoremModP}
For $a\in\Z/p[FX]$, $a\in \Delta^n_p$ iff $\forall \, k, \ \forall \, (i_1\dots i_k)$
$$
\e\partial^k_{i_1\dots i_k}(a)\equiv 0 \mod p
$$
\end{cor}
\begin{proof}
Corollary follows from the mod-$p$ version of Magnus embedding
\begin{equation*}
\mu_p:FX\to\Z/p\dlg X\drg, \ x\mapsto x+1
\end{equation*}
which is the composition of $\mu$ and the quotient map $\Z\dlg X\drg\to\Z/p\dlg X\drg$. Therefore, coefficients of elements in the image of $\mu_p$ are given by augmentations of Fox derivatives modulo prime $p$. And as before, after extending $\mu_p$ to the group ring $\Z/p [FX]$, powers of augmentation ideals $\Delta_p^n$ will maps to a elements of length $\leq p$. These two facts give the result.
\end{proof}

Theorem (\ref{FoxTheorem}) together with lemma (\ref{lowerCentralAugmentationLemma}) provides a convenient way to check whenever $H_k$ preserves weighted lower ($p$-)central series filtration.
\begin{thm}\label{JHLowerCentral}
Simplicial James-Hopf map $H_m: F[X]\to F[X^{\w m}]$ sends lower central series to a weighted one:
$$
H_m(\g_n)\subset \g^w_n, \ H_m(\gp_n)\subset \g^{[p],w}_n
$$
\end{thm}
\begin{proof}
By lemma (\ref{lowerCentralAugmentationLemma}) it is sufficient to prove preservation of weighted augmentation ideal filtration, i.e.  that $\forall g_1,\dots g_m\in F[X]$ $H_n((g_1-1)\dots (g_m-1))\in \D^m_w$. After writing each $g_i$ in normal form, the statement can be reduced to the case when all $g_i$ are generators of $F[X]$, possibly with negative exponent. Since $H_n$ is natural with respect to maps between sets of generators $Y\to X$, all this generators can be taken distinct from each other and ordered lexicographically.
\begin{equation*}
\xym{
\Z F[Y]\ar[r]^{\Z H_n}\ar[d]_f&\Z F[Y^{\w n}]\ar[d]^{f\w\dots \w f}\\
\Z F[X]\ar[r]_{\Z H_n}&\Z F[X^{\w n}]
}
\end{equation*}
So 
\begin{equation*}
\Z H_n((x_{i_1}^{\e_{i_1}}-1)\dots (x_{i_m}^{\e_{i_m}}-1))=(f\w\dots\w f)\Z H_n((y_1^{\e_1}-1)\dots(y_m^{\e_m}-1))
\end{equation*}
According to the theorem (\ref{FoxTheorem}) it is sufficient to show that augmentations of all derivatives $\partial^i_{J_i,\dots, J_1}$ of $\Z H_n((y_1^{\e_1}-1)\dots(y_m^{\e_m}-1))$ of orders up to $q$ (or $q-1$) are vanished for all possible combinations of multi-indices $J_k=y_{k_1}\w\dots\w y_{k_n}$.

Let $Y=y_1^{\e_1}\dots y_m^{\e_m}$, $|v|$ is the length of word $v$ and notations $s\subset v, \ s\subseteq v$ will stands for $s$ being a subword of word $v$ and $s$ being a subword (with possible repetitions) of word $v$ respectively. With these notations
\begin{equation}\label{der}
\varepsilon\partial^i_{J_i,\dots, J_1}\Z H_n((y_1^{\e_1}-1)\dots(y_m^{\e_m}-1))=\sum_{j=0}^{m}(-1)^{m-j}\sum_{v\subset Y, \ |v|=j}\varepsilon(\partial^i_{J_i,\dots , J_1}H_n(v))
\end{equation}
According to formula (3.5) of \cite{FoxFDCI}, in which higher derivatives of arbitrary reduced word in free group are computed, and using formula for $H_n(v)=\prod_{d\subseteq v}d^{\e_d}$ we see that

\begin{equation}\label{test}
\e\partial^i_{J_i,\dots, J_1}(H_n(v))=\sum \e_{d_1}\dots\e_{d_i}
\end{equation}

with sum over all sequences of subwords (with possible repetitions) $d_k$ of $v$, such that for all subwords $d_k$ in sequence $J_k=d_k$ and subword $d_k$ can repeat if the corresponding exponent $\e_{d_k}$ is negative.

This means that augmentation of derivative is not zero if corresponding $v$ contains all subwords $J_k$ (again, repetitions are allowed if corresponding exponents are negative). First of all, since everything in $v$ is ordered lexicographically, all derivatives with respect to tuples $J_q\dots J_1$, which are not in (anti)lexicographic order, are zero. Second, since all elements in $Y$ are different, all subwords $d\subseteq v$ are different from each other, so all $J_k$ appears in $H_n(v)$ only once, with fixed exponent $\e_{J_k}$ and (\ref{der}) will take form 
\begin{equation*}
\varepsilon\partial^i_{J_i,\dots, J_1}\Z H_n((y_1^{\e_1}-1)\dots(y_m^{\e_m}-1))=\e_{J_i}\dots\e_{J_1}\sum_{j=0}^{m}(-1)^{m-j}\sum_{v\subset Y, \ |v|=j, \ J_i\dots J_1\subseteq v}1
\end{equation*}
Next we count the number of subwords $v\subset Y$, of length $j$ which contains $J_i\dots J_1\subseteq v$, for fixed tuple $J_i\dots J_1$. If generator $y_l$ have a positive exponent in $Y$ then derivative with respect to a tuple, which contains multi-index $J_p=y_{p_1}\w\dots\w y_l\w y_l\w\dots\w y_{p_n}$ with repetition of $y_l$ will be zero. So we can assume that all $J_k$ have repetition of indexes inside only if the exponent of corresponding generator is negative in $Y$. Same goes for the case when $J_k=J_{k+1}$ - such derivatives a zero if the product of exponents of generators in $Y$ which also form $J_k$ is positive (does not depend on $v$). With this limitations, the number of $v$ of length $j$ such that $J_i\dots J_1\subseteq v$ is precisely
\begin{equation*}
\binom{m-n_1-\dots-n_r}{j-n_1-\dots-n_r}
\end{equation*}
Here $r$ is a total number of different $J_k$ ($1\leq r\leq i$) and $n_k$ is a total number of different generators in $J_k$ ($1\leq n_k\leq n$). Denoting $N=n_1+\dots n_r$ and substituting this number in the formula we have
$$
\varepsilon\partial^i_{J_i,\dots, J_1}\Z H_n((y_1^{\e_1}-1)\dots(y_m^{\e_m}-1))=\e_{J_i}\dots\e_{J_1}\sum_{j=N}^{m}(-1)^{m-j}\binom{m-N}{j-N}=
$$
$$
=\begin{cases} 
      0 & m-N\neq 0\\
      \e_{J_i}\dots\e_{J_1} & m-N=0
   \end{cases}
$$
since $N\leq i n$ , $s<n$ and $i\leq q$, last case can only happen if $s=0, N=qn$ which is exactly case 1 in the definition of weighted filtration.
\end{proof}
\begin{rem}
Passi (see \cite{PassiPaper}, \cite{PassiBook}) defined a map $f$ between group $G$ and $R$-module $M$ to be \textit{polynomial of degree $\leq n$} if its extension on a group algebra $R[G]$ vanishes on $\Delta^{n+1}_R$. Therefore theorem (\ref{JHLowerCentral}) states that maps $\varepsilon \partial^j_{J_j\dots J_1}\Z H_k:\Z F[X]\to \Z$ are $\Z$-polynomial of degree $\leq jk$.
\end{rem}
\begin{cor}
James-Hopf invariants $H_k$ sends lower $p$-central series to a weighted ones:
\begin{equation*}
H_k(\gp_n)\subset \gamma_n^{[p],w}
\end{equation*}
\end{cor}
\begin{proof}
As before we switch our attention to the powers of augmentation ideal $\Delta_p$, by corollary (\ref{FoxTheoremModP}) we need to show that augmentations of corresponding Fox derivatives of $H_k$ vanishes mod $p$. This follows from proposition (1.11) in \cite{PassiBook} which states that composition of $\Z$-polynomial map $f:G\to A$ and homomorphism $\theta: A\to M$ for some $R$-module $M$ is a $R$-polynomial map. Application of this proposition to $\varepsilon \partial^j_{J_j\dots J_1} H_k$ and quotient map $\Z\to\Z/p$ gives the result.
\end{proof}
\begin{thm}\label{actionOnSS}
Elements of Cohen group $H_{\infty}(X)$ induce maps of towers $\Gamma\to \Gamma$ and $\Gamma_p\to \Gamma_p$, and hence equip spectral sequences (\ref{spectralSeq}) with an action of $H_{\infty}: E^r_{p,q}\to E^r_{p,q}$ which preserves differential.
\end{thm}
\begin{proof}
Combining lemma (\ref{WhiteheadLowerCentral}) with theorem (\ref{JHLowerCentral}) we see that generators of Cohen group (and hence all elements) preserves lower ($p$-)central series filtration, and it is immediate that inclusion $F[-]\to F^{\w}[-]$ induce an isomorphisms
\begin{equation}\label{LCSquotientIso}
F[-]/\gamma_n\cong F^{\w}[-]/\gamma_n, \ F[-]/\gp_n\cong F^{\w}[-]/\gp_n, 
\end{equation}
so there are a well defined maps
\begin{align}\label{CohenGroupToTowersMap}
H_{\infty}\to \hom(\Gamma,\Gamma), \ f:F[-]\to F^{\w}[-] \mapsto \{f_n\}_n: \{F[-]/\gamma_n\}_n\to \{F[-]/\gamma_n\}_n\\
H_{\infty}\to \hom(\Gamma_p,\Gamma_p), \ f:F[-]\to F^{\w}[-] \mapsto \{f_n\}_n: \{F[-]/\gp_n\}_n\to \{F[-]/\gp_n\}_n
\end{align} 
Note here that $\hom(\Gamma,\Gamma)$ and $\hom(\Gamma_p,\Gamma_p)$ are groups with respect to pointwise multiplication on each layer of tower, and therefore (\ref{CohenGroupToTowersMap}) are homomorphisms.
here $f_n$ denotes induced map on quotients $F[-]/\gamma_n$ (or $F[-]/\gp_n$). Image of these maps will be called \textit{Cohen groups for towers} $\Gamma$ \textit{and} $\Gamma_p$, denoted by $H\Gamma_{\infty}$ and $H\Gamma_{p,\infty}$

Action of $H_{\infty}$ on spectral sequences comes from the corresponding action on towers through an induced action on exact couples in a standard way. Each $f\in H_{\infty}$ gives a map $E^1_{p,q}\to E^1_{p,q}$ which is an induced map on homotopy groups of fibers. Preservation of differential follows from the fact that $f$ induce a map of towers.
\end{proof}
\section{Functorial decomposition of lower $p$-central series spectral sequence}
Through this section $\k=\Z/p$ will be ground field with $p$ elements, $V$ be a $\k$-module. First we will review few facts about natural coalgebra transformations $T\to T$. Tensor algebra $T(V)=\bigoplus_{n=0}^{\infty} V^{\ot n}$ can be considered as a (connected) Hopf algebra (see \cite{MilnorMoore}) with the standard multiplication $m$ and comultiplication $\psi$ defined as
\begin{equation*}
\psi(v)=v\ot 1 + 1\ot v, \ v\in V
\end{equation*}
on generators (making them primitive) and extended to all $T(V)$. Conjugation $s$ defined as
\begin{equation*}
s(x_1\dots x_k)=(-1)^k x_k\dots x_1
\end{equation*}

Using this comultiplication and conjugation one can define a group structure on set of all coalgebra maps $\hom_{\mathrm{Coalg}_{\k}}(T(V),T(V))$, using so-called \textit{convolution product}: for $f,g:T(V)\to T(V)$ $f*g$ defined as composition
\begin{equation}\label{convolutionProduct}
T(V)\xrightarrow{\psi}T(V)\ot T(V)\xrightarrow{f\ot g}T(V)\ot T(V)\xrightarrow{m} T(V)
\end{equation}
Note that this product is a complete analogue of point-wise multiplication on set of all maps $G\to G$ for group $G$:
\begin{equation*}
G\xrightarrow{\Delta}G\times G \xrightarrow{f\times g} G\times G \xrightarrow{\mu} G
\end{equation*}

$T$ as coalgebra naturally filtered by \textit{James filtration}:
\begin{equation*}
J_k=\bigoplus_{n=0}^k T_n
\end{equation*}
here $T_n$ stands for a functor $T_n(V)=V^{\otimes n}$ and this filtration is complete and gives $\hom_{\mathrm{Coalg}_{\k}}(T(V),T(V))$ structure of a pro-group. 
\begin{lem}[(2.4) in \cite{SelickWuCoalg}]
there is a short exact sequence of groups
\begin{equation}\label{SESHomProGroup}
\xym{
\hom_{\k}(T_n,\L^n)\cong\ker\mono[r]&\hom_{\mathrm{Coalg}_{\k}}(J_n,T)\epi[r]&\hom_{\mathrm{Coalg}_{\k}}(J_{n-1},T)
}
\end{equation}
\end{lem}
For a simplicial realization theorem below we will use a specific set of topological generators of a pro-group $\hom_{\mathrm{Coalg}_{\k}}(T,T)$, first apparently appeared in \cite{CohenTaylor}:
\begin{thm}\label{tensorCoalgebraGeneratorsTheorem}
Any natural coalgebra self-map $f_V:T(V)\to T(V)$ (evaluated at $V$) can be written down as a (possibly infinite) convolution product $\prod_{k=1}^{\infty} f_k$ of the maps of the form
\begin{equation}\label{tensorCoalgebraGenerators}
T(V)\xrightarrow{H_k^{alg}}T(V^{\otimes k})\xrightarrow{T(\sigma)}T(V^{\otimes k})\xrightarrow{\beta^T_k}T(V)
\end{equation}
here $H_k^{alg}$ is an \textit{algebraic James-Hopf map} (proposition (5.3) in \cite{CohenTaylor}), $\sigma\in\Sigma_k$ and $\beta^T_k$ is a Hopf extension of a map $\beta_k: V^{\otimes k}\to T(V), \ \beta_k(x_1\dots x_k)=[x_1\dots x_k]$.
\end{thm}
\begin{proof}
It is sufficient to show the statement for each element in inverse system. The statement is trivial for maps in $\hom_{\mathrm{Coalg}_{\k}}(J_1,T)$ and suppose that any element $f\in\hom(J_{k-1},T)$ can be expressed as 
\begin{equation*}
f=\prod_{j=1}^{k-1} f_j
\end{equation*}
Given $f:J_k\to T$ it can be decomposed in a convolution product $f_1*f_2$, where $f_1:J_k\to T$ is a coalgebra map which is an identity map on the $(J_k)_0=\k$, trivial on $(J_k)_i, \ 0<i<k$ and $(f_1)_k=(f)_k$ and $f_2$ is the restriction of $f$ to $J_{k-1}$. Therefore it is sufficient to express $f_1$ as a convolution product of the maps of the form (\ref{tensorCoalgebraGenerators}). We will use universal property of the algebraic James-Hopf map which is direct analogue of the universal property of the classical combinatorial James-Hopf map:
\begin{equation}\label{algebraicJHUniversalProperty}
\xym{
\top :&&&\mathrm{Coalg} :&&\\
&J_k X \epi[d]\mono[r]& JX\ar@{-->}[d]^{H_k}&&J_k V \epi[d]\mono[r]& T(V)\ar@{-->}[d]^{H^{alg}_k}\\\
&X^{\w k}\mono[r]& JX^{\w k}&&V^{\otimes k}\mono[r]& T(V^{\otimes k})
}
\end{equation}
Applying it to $f_1$ (which factors through $T_k$ by construction) we see that
\begin{equation*}
\xym{
J_k\mono[rr]\epi[dd]\ar[dr]^{f_1}&&T\ar[dd]^{H_k^{alg}}\\
&T&\\
T_k\ar[ur]^{\bar f_1}\mono[rr]&&T(T_k)\ar[ul]^{T(\bar f_1)}
}
\end{equation*} 
Since $\hom_{\k}(T_k, T_l)=\k[\Sigma_k]$ if $k=l$ and $0$ otherwise, the map $\bar f_1:T_k\to T$ is completely determined by a certain element $\sum_i a_i\sigma_i$ of a group algebra of the symmetric group $\Sigma_k$. Moreover, since $\bar f_1$ came initially from a coalgebra map, its image lies in $\L_{res}^k$ (see proposition (2.4) in \cite{SelickWuCoalg}) and therefore it factors through $\beta_k$. Finally under isomorphism in (\ref{SESHomProGroup}) sum of permutations maps to a product and we have
\begin{equation*}
f_1=\prod_i \beta^T_k\circ\sigma_i^{a_1}\circ H_k
\end{equation*}
\end{proof}
\begin{rem}
It can be noticed that in contrast with definition (\ref{CohenGroupGenerators}) there are no reduced diagonals in generators of $\hom_{\mathrm{Coalg}_{\k}}(T,T)$ - these maps are trivial on tensor algebra.
\end{rem}

Similarly to l.c.s. tower, using powers of augmentation ideal filtration on a group algebra $\k [F[-]]$ one can form a tower of simplicial abelian groups (denoted by $\k\Gamma_p$)
\begin{equation}\label{AugmentationTower}
\xym{
\k[F]^{\w}=\varprojlim \k[F]/\Delta^n_p\ar[d]&\\
\vdots\ar[d]&\\
\k[F]/\Delta^{n+1}_p\epi[d]&T_n(\bar\k[-])\mono[l]\\
\k[F]/\Delta^n_p\epi[d]&\\
\vdots\epi[d]&\\
\k[F]/\Delta^2_p=\k\oplus\bar\k[-]&
}
\end{equation}
with fibers identified with a homogeneous components of tensor algebra functor $T(\bar \k[-]):\sset\to \mathrm{sMod}_{\k}$ (analogue of Magnus-Witt theorem, see \cite{CurtisSimpl})
\begin{equation*}
T_n(\bar \k [-])\xrightarrow{\cong}\Delta^n_p/\Delta^{n+1}_p, \ x_1\dots x_n\mapsto (x_1-1)\dots (x_n-1)+\Delta^{n+1}_p
\end{equation*}
The natural inclusion $F[-]\to\k [F[-]]$ of Milnor's construction to its group algebra can be extended to a level of towers
\begin{equation*}
F[-]/\gp_n\to \k [F[-]]/\Delta^n_p
\end{equation*}
in a way that on the level of fibers one get inclusion $\L_{res}^n\to T_n$ of homogeneous components of free restricted Lie algebra to homogeneous components of tensor algebra as primitive elements in degree $n$ (see also \cite{MilnorMoore}).

As before $E^0$ will denote the functor of fibers of tower, in case of towers (\ref{towers}) and (\ref{AugmentationTower}) it takes values in graded objects of $\mathrm{sMod}_{\k}$. Therefore there is a map
\begin{equation}\label{realizationMap}
\hom_{\sset}(\Gamma_p,\Gamma_p)\xrightarrow{E^0\circ \k[-]}\hom_{\mathrm{sCoalg}_{\k}}(T(\bar \k[-]),T(\bar \k[-]))
\end{equation}
\begin{thm}\label{simplicialRealizationTheorem}
(\ref{realizationMap}) is a well-defined epimorphism of groups, i.e. every coalgebra natural transformation of functor $T(\bar \k[-])$ to itself can be lifted up to a natural self-transformation of tower $\Gamma_p\to \Gamma_p$
\end{thm}
\begin{proof}
First we will show that $E^0\circ \k$ is a well defined homomorphism, i.e. for any $f:\Gamma_p\to\Gamma_p$ $E^0(\k[f])$ is a coalgebra map and products are preserved. To see this we pass to the limits of the towers:
\begin{equation*}
\xym{
\hom(\Gamma_p,\Gamma_p)\ar[d]_{\varprojlim}\ar[r]^{\k}&\hom(\k\Gamma_p,\k\Gamma_p)\ar[d]^{\varprojlim}\ar[r]^{E^0}&\hom(T,T)\\
\hom(F^{\w}_p[-],F^{\w}_p[-])\ar[r]_{\varprojlim\k}&\hom(\k[F]^{\w},\k[F]^{\w})\ar[ur]_{E^0}
}
\end{equation*}

here we consider $\k[F]^{\w}$ as a complete coalgebra with a filtration, induced from $\Delta_p$-adic filtration on $\k[F]$ and $E^0$ functor in this context is a functor of passing to associate graded coalgebra. $E^0$ respects (completed) tensor products:
\begin{equation}\label{E0Products}
E^0(-\hat \otimes -)=E^0(-\otimes -)=E^0(-)\otimes E^0(-)
\end{equation}

$\varprojlim\k$ is an extension of group algebra functor $\k[-]:\mathrm{Grp}\to\mathrm{Hopf}_{\k}$ to completions and therefore for any map $g:F^{\w}_p\to F^{\w}_p$ $\varprojlim\k (g)$ is a map of complete coalgebras. This shows that the map (\ref{realizationMap}) is well-defined.

By (\ref{E0Products}) and since $\varprojlim$ commutes with direct products and $\varprojlim \k (F^{\w}_p\times F^{\w}_p)=\k F^{\w}\hat\otimes \k F^{\w}$ $E^0\circ \k$ sends pointwise multiplication 
\begin{equation*}
F/\gp_n\xrightarrow{\Delta}F/\gp_n\times F/\gp_n\xrightarrow{f\times g}F/\gp_n\times F/\gp_n\xrightarrow{\mu} F/\gp_n
\end{equation*}
to the convolution product (\ref{convolutionProduct}).

To prove that $E^0\circ \k[-]$ is surjective we will show that generators of Cohen group for towers maps to generators of $\hom_{\mathrm{sCoalg}_{\k}}(T(\bar \k[-]),T(\bar \k[-]))$, which are described in (\ref{tensorCoalgebraGeneratorsTheorem}). It is clear that $E^0(\k[F[\sigma]])=T(\sigma)$ and $E^0(\k[W_n])=\beta^T_n$. Only universal property (\ref{algebraicJHUniversalProperty}) of $H^{alg}_k$ is used in theorem (\ref{tensorCoalgebraGeneratorsTheorem}) and $E^0(\k[H_k])$ satisfy it since $E^0(\k[J])\equiv E^0(\k[F])$ - associated graded of free group ring and free monoid ring are naturally isomorphic (isomorphism induced by inclusion $J\to F$). The result follows from diagrams (\ref{algebraicJHUniversalProperty}) and (\ref{connectionWithClassicalJH}).
\end{proof}

We proceed to a decomposition of spectral sequence (\ref{spectralSeq}). First we translate the well-known result about homotopy idempotents of $H$-spaces to $\sset$.
\begin{lem}\label{decompositionLemma}
Let $G$ be a connected simplicial group and $f: G\to G$ be a simplicial self-map, such that
\begin{equation*}
f_*:\pi_* G\to\pi_* G
\end{equation*}

is an idempotent. Then, as simplicial set, $G$ is weak equivalent to a product $A\times B$, 
\begin{equation*}
A=\mathrm{colim} \,f, \ B=\mathrm{colim} \,g
\end{equation*}
where $g=[\mathrm{id}]f^{-1}$ denotes a complement to $f$ in the group $\hom_{\sset}(G,G)$
\end{lem} 
\begin{proof}
$f_*$ is a (graded) idempotent element of $\mathrm{End}(\bigoplus_i \pi_i G)$ and  therefore there is an isomorphism
\begin{equation}\label{decompositionSpaceHomotopyLevel}
\bigoplus_i \pi_i G \equiv \bigoplus_i \im{f_i}\oplus \im(\mathrm{id}-f_i)
\end{equation} 
here $\im{f_i}=\mathrm{colim}{f_i}=\pi_i \mathrm{colim} f, \ \im(\mathrm{id}-f_i)=\mathrm{colim}(\mathrm{id}-f_i)=\pi_i \mathrm{colim} [\mathrm{id}][f^{-1}]$
and isomorphism (\ref{decompositionSpaceHomotopyLevel}) can be realized by
\begin{equation*}
G\xrightarrow{\Delta}G\times G\to A\times B
\end{equation*}
\end{proof}
Well known results about filtered colimits and fibrations in $\sset$ will be organized in the next lemma for simplicity.
\begin{lem}\label{colimitLemma}
Let $X_0\to X_1 \to\dots\to X$ be a sequence of simplicial sets and simlplicial maps between them with colimit $X$ and
$$
\xym{
X_0\ar[r]\ar[d]^{f_0}&X_1\ar[r]\ar[d]^{f_1}&\dots\ar[r]&X\ar[d]^f\\
Y_0\ar[r]&Y_1\ar[r]&\dots\ar[r]&Y
}
$$
be a map of such sequences. Then:
\begin{enumerate}[\upshape(i)]
\item $\pi_* X=\mathrm{colim}\,\{\pi_*(X_0)\to\pi_*(X_1)\to\dots\}$ 
\item if $f_n$ are Kan fibrations for all $n$ then $f$ is also Kan fibration.
\item the fiber of $f$ is a colimit of fibers of $f_n$
\end{enumerate}
\end{lem}

Given arbitrary natural coalgebra decomposition of tensor algebra $T\simeq A\ot B$, it can be extended to $\sset$ by switching from functor $T:\mathrm{Mod}_{\k}\to\mathrm{Coalg}_{\k}$ to $T(\bar\k [-]):\sset\to \mathrm{sCoalg}_{\k}$. For any $X\in\sset$ composition of retraction and inclusion
\begin{equation*}
f:T(\bar\k[X])\xrightarrow{r_A} A(\bar\k[X])\xrightarrow{i_A} T(\bar\k[X])
\end{equation*}
is an idempotent in $\hom_{\mathrm{sCoalg}_{\k}}(T(\bar\k[X]),T(\bar\k[X]))$ and $A(\bar\k[X])\simeq\mathrm{colim} f$. This idempotent give rise to a map of towers $\varphi:\Gamma_p\to \Gamma_p$ by simplicial realization theorem (\ref{simplicialRealizationTheorem}), on each level $n$ denoted by $\varphi_n: F[X]/\gp_n\to F[X]/\gp_n$. 
\begin{thm}\label{towerDecomposition}
For any natural (simplicial) coalgebra decomposition $T\cong A\otimes B$ there exists towers of fibrations $\EuScript A$ and $\EuScript B$ such that
\begin{equation*}
\Gamma_p\simeq \EuScript A\times\EuScript B
\end{equation*}
and 
such that its fibers are functorially given by primitive elements of simplicial coalgebras $A$ and $B$:
\begin{equation*}
E^0(\EuScript A)\simeq PA, \ E^0(\EuScript B)\simeq PB
\end{equation*}
\end{thm}
\begin{proof}
Given an $X\in \sset$, and decomposition $T(\k[X])\simeq A(\k[X])\ot B(\k[X])$ let $f_X: T(\k[X])\to T(\k[X])$ be an idempotent, corresponding to $A(\k[X])$, as before, its realization on the level of towers is $\varphi_X:\Gamma_p(X)\to\Gamma_p(X)$. Since in each dimension $k$ $X_k$ is a finite set, for any $n$ $(F[X]/\gp_n)_k$ are finite groups. Hence, after sufficiently many iterations $\varphi_{X,n}:F[X]/\gp_n\to F[X]/\gp_n$ will become an idempotent, i.e. there exist number $N(n,k)$, such that 
\begin{equation}\label{idempotentMap}
(\varphi_{X,n})_k^{N(n,k)}:(F[X]/\gp_n)_k\to (F[X]/\gp_n)_k
\end{equation}
 are idempotents. Since $\varphi_X$ was initially map of towers, $\varphi_X^N$ is also one, which follows from commutativity of the following diagrams:
\begin{align*}
\xym{
(F[X]/\gp_n)_k\ar@<.5ex>[r]\ar@<2ex>[r]_{\dots}\ar[d]_{(\varphi_{X,n})_k^{N(n,k)}}&(F[X]/\gp_n)_{k-1}\ar@<.5ex>[l]\ar@<2ex>[l]_{\dots}\ar[d]^{(\varphi_{X,n})_k^{N(n,k-1)}}\\
(F[X]/\gp_n)_k\ar@<.5ex>[r]\ar@<2ex>[r]_{\dots}&(F[X]/\gp_n)_{k-1}\ar@<.5ex>[l]\ar@<2ex>[l]_{\dots}
} & 
\xym{
F[X]/\gp_n\ar[rr]^{(\varphi_{X,n})_k^{N(n)}}\epi[d]&&F[X]/\gp_n\epi[d]\\
F[X]/\gp_{n-1}\ar[rr]_{(\varphi_{X,n-1})_k^{N(n-1)}}&&F[X]/\gp_{n-1}
}\\
\end{align*}
Now lemma (\ref{decompositionLemma}) is applied to groups $F[X]/\gp_n$ and idempotents (\ref{idempotentMap}) to get a decompositions 
\begin{equation}\label{decompositionOneLevel}
F[X]/\gp_n\simeq \EuScript A_n(X)\times \EuScript B_n(X), \ \EuScript A_n(X)=\mathrm{colim}\, \varphi_{X,n}^{N(n)}=\mathrm{colim}\, \varphi_{X,n}
\end{equation}
This weak equivalence is consistent with a maps $\xym{F[X]/\gp_n\epi[r]&F[X]/\gp_{n-1}}$ and by lemma (\ref{colimitLemma}) the natural maps $\EuScript A_n(X)\to \EuScript A_{n-1}(X)$ are fibrations, similarly for $\EuScript B_n(X)$, so (\ref{decompositionOneLevel}) gives indeed decomposition of the whole tower $\Gamma_p (X)$, and this decomposition is functorial with respect to $X$.

By commutativity of the following diagram 
\begin{equation*}
\xym{
\hom(\Gamma_p,\Gamma_p)\ar[r]^{\k}\ar[d]_{E^0}&\hom(\k\Gamma_p,\k\Gamma_p)\ar[d]^{E^0}\\
\hom(\L,\L)&\hom(T,T)\ar[l]^P
}
\end{equation*}
and lemma (\ref{colimitLemma}):
\begin{equation*}
E^0\EuScript A =E^0 \mathrm{colim} \,\varphi\simeq \mathrm{colim}\, E^0 \varphi= \mathrm{colim} \,Pf=P (\mathrm{colim}\, f)=P A
\end{equation*} 
\end{proof}
\begin{rem}
After passing to the limits of towers one have functorial decomposition of the pro-p completion of Milnor's construction $F[-]$
\begin{equation}\label{limitDecomposition}
F^{\w}_p\simeq \varprojlim \EuScript A\times \varprojlim \EuScript B
\end{equation}
and since $|F^{\w}_p[X]|\simeq_p \Omega\Sigma |X|$ (\ref{limitDecomposition}) can be seen as a translation of the classical decomposition (\ref{classicalDecomposition}) to the simplicial setting.
\end{rem}
\begin{cor}\label{SSDecompositionCorollary}
Any natural coalgebra decomposition $T\simeq A\ot B$ induce a decomposition of spectral sequence
\begin{equation}\label{SSDecomposition}
E^1_{p,q}=\pi_p \L_{res}^q\ra \pi_{p+q} F, \ E^r_{p,q}=E^r_{p,q}(A)\oplus E^r_{p,q}(B)
\end{equation}
as a functor on $\sset$, with first pages of $E^r_{p,q}(A),\, E^r_{p,q}(B)$ given by homotopy groups of primitive elements of simplicial coalgebras $A$ and $B$:
\begin{equation*}
E^1_{p,q}(A)=\pi_p (PA)_q, \ E^1_{p,q}(B)=\pi_p (PB)_q
\end{equation*}
\end{cor}

In conclusion we will formulate the result about accelerated functorial sub-spectral sequence of the unstable Adams spectral sequence for suspensions, which is the direct corollary of (\ref{SSDecompositionCorollary})
\begin{cor}\label{acceleration}
Let $A^{min}$ be a minimal functorial coalgebra retract of $T$ and $T=A^{min}\otimes B^{max}$ be a corresponding decomposition. Then sub-spectral sequence $E^r_{p,q}(\EuScript{A}^{min})$ of $E^r_{p,q}$ have non-trivial cells only in columns with numbers $p^k$
\end{cor}

\bigskip
\footnotesize
\noindent\textit{Acknowledgments.}
The research partially supported by the Singapore Ministry
of Education research grant (AcRF Tier 1 WBS No. R-146-000-222-112). The second author is also partially supported by a grant (No. 11329101) of NSFC of China.

\bibliographystyle{gtart}

\end{document}